\documentclass[11pt]{amsart}
\usepackage{amsmath, amsfonts, amsthm, amssymb, multicol, mathtools, dsfont, verbatim}
\usepackage{graphicx}
\usepackage{float, hyperref}
\usepackage{scalerel,stackengine, subcaption}
\usepackage[usenames,dvipsnames]{xcolor}
\usepackage{enumitem}
\usepackage{pgfplots}
\usepgfplotslibrary{fillbetween}
\pgfplotsset{width=10cm,compat=1.9}
\usepackage[square,sort,comma,numbers]{natbib}
\allowdisplaybreaks

\usepackage[square,sort,comma,numbers]{natbib}
\usepackage{fancyhdr}

\makeatletter
\def\@setauthors{%
  \begingroup
  \def\thanks{\protect\thanks@warning}%
  \trivlist
  \centering\footnotesize \@topsep30\p@\relax
  \advance\@topsep by -\baselineskip
  \item\relax
  \author@andify\authors
  \def\\{\protect\linebreak}

  \normalsize\lowercase{\authors}%
  
	\ifx\@empty\contribs
  \else
    ,\penalty-3 \space \@setcontribs
    \@closetoccontribs
  \fi
  \endtrivlist
  \endgroup
}
\def\@settitle{\begin{center}
\LARGE\lowercase{\@title}
  \end{center}%
}
\makeatother

\definecolor{lightblue}{HTML}{2B77A4}
\definecolor{darkred}{HTML}{9E0D0D}
\hypersetup{
	colorlinks=true,
	linkcolor=darkred,
	urlcolor=darkred,
	citecolor=lightblue
}
\numberwithin{equation}{section}

\newtheorem{thm}{Theorem}[section]

\newtheorem{cor}[thm]{Corollary}

\newtheorem{prop}[thm]{Proposition}

\renewcommand{\epsilon}{\varepsilon}
\newcommand{\eps}{\varepsilon}

\newcommand{\rd}{\mathbb{R}^d}

\renewcommand{\geq}{\geqslant}
\renewcommand{\leq}{\leqslant}

\newcommand{\ubd}{\overline{\dim}_{\textup{B}}}
\newcommand{\lbd}{\underline{\dim}_{\textup{B}}}
\newcommand{\hd}{\dim_{\textup{H}}}
\newcommand{\pd}{\dim_{\textup{P}}}
\newcommand{\ad}{\dim_{\textup{A}}}

\newcommand{\fd}{\dim_{\mathrm{F}}}

\title{Brascamp--Lieb inequalities for fractal dimensions}

\author{Jonathan M. Fraser\\ \\
 University of St Andrews, Scotland\\
\MakeLowercase{Email: jmf32@st-andrews.ac.uk}}
\thanks{The  author was  financially supported by a  \emph{Leverhulme Trust Research Project Grant} (RPG-2023-281)  and an \emph{EPSRC Open Fellowship} (EP/Z533440/1).}

\begin{document}


\maketitle
\thispagestyle{empty}

\begin{abstract}
We use the Brascamp--Lieb inequality from functional analysis to prove   novel inequalities  for the upper box,  packing, and Assouad  dimensions of fractal sets in terms of the dimensions of certain projections.  Analogous inequalities do not hold for Hausdorff  or lower box dimensions. We apply these fractal Brascamp--Lieb inequalities   to establish new exceptional set estimates for orthogonal projections  and to provide sharp dimension estimates for certain constrained sumsets. We also establish analogous   nonlinear inequalities via the nonlinear Brascamp--Lieb inequality.
\end{abstract}

\vspace{1cm}

\tableofcontents

\newpage

\section{The Brascamp--Lieb inequality}

The Brascamp--Lieb inequality is a fundamental inequality in modern analysis with many varied applications.  In this paper we use it to establish novel inequalities for various fractal dimensions in terms of the dimensions of certain projections.  Surprisingly, these rather fundamental inequalities seem not to have  been formulated before.

We first recall the  Brascamp--Lieb inequality. Let $m,d \geq 1$ be natural numbers and, for each $i=1, \dots, m$, let $d_i \leq d$ be a natural number and $c_i>0$ be a weight such that
\begin{equation} \label{scaling}
\sum_{i=1}^mc_i d_i = d.
\end{equation}  
Further, let $P_i: \rd \to \mathbb{R}^{d_i}$ be an associated collection of linear surjections and  assume that for all subspaces $V$ of $\rd$
\begin{equation} \label{dimension}
\dim V \leq \sum_{i=1}^m c_i \dim \left(P_i(V)\right).
\end{equation} 
Observe that \eqref{dimension} ensures that   $\cap_{i=1}^m\textup{ker}(P_i) = \{0\}$ since otherwise setting $V= \cap_{i=1}^m\textup{ker}(P_i)$ contradicts \eqref{dimension}. The conditions \eqref{scaling} and \eqref{dimension} are often referred to as the \emph{scaling} and \emph{dimension} conditions, respectively.  The  \emph{Brascamp--Lieb inequality}  states that there is a finite constant $C>0$, depending only on the above \emph{datum} $(\{P_i\}_i, \{c_i\}_i)$, such that for all non-negative $f_i \in L^1(\mathbb{R}^{d_i})$ 
\begin{equation} \label{bli}
\int_{\rd} \prod_{i=1}^m f_i\left(P_i(x)\right)^{c_i} \, dx \leq C\  \prod_{i=1}^m\left(\int_{\mathbb{R}^{d_i}}  f_i(y) \, dy\right)^{c_i}.
\end{equation}  
This version is due to Bennett,   Carbery, Christ, and Tao \cite{bcct}.   The Brascamp--Lieb inequality was first formulated  in \cite{bl} and  generalises the Loomis--Whitney inequality \cite{lw},    H\"older's inequality, and the Poincar\'e inequality. It  is an important tool in modern  harmonic analysis and   functional analysis, and is widely used in convex geometry, information theory and PDEs.

We will also be interested in the following nonlinear Brascamp--Lieb inequality due to Bennett, Bez, Buschenhenke, Cowling, and Flock \cite{bbbcf}.  Let $P_i, c_i$ for $i=1, \dots, m$ be as above and let $T_i: \rd \to \mathbb{R}^{d_i}$ be $C^2$ submersions in a neighbourhood of a point $x_0$ such that $J_{x_0}(T_i) = P_i$ for all $i=1, \dots, m$, where $J_{x_0}(T_i)$ is the Jacobian derivative of $T_i$ at $x_0$.  Then, for all $\eps>0$, there exists a neighbourhood $U$ of $x_0$ such that
\begin{equation} \label{nbli}
\int_{U} \prod_{i=1}^m f_i\left(T_i(x)\right)^{c_i} \, dx \leq (1+\eps)C\  \prod_{i=1}^m\left(\int_{\mathbb{R}^{d_i}}  f_i(y) \, dy\right)^{c_i}.
\end{equation}  
Here $C$ is the same constant as in \eqref{bli}, namely, the Brascamp--Lieb constant associated with  the Brascamp--Lieb datum $(\{P_i\}_i, \{c_i\}_i)$.
 
\section{Main results:~Brascamp--Lieb inequalities for fractal dimensions}

\subsection{Upper box dimension}

Write $\mathcal{L}^d$ for the $d$-dimensional Lebesgue measure on $\rd$ and let $X \subseteq \rd$ be a non-empty bounded set. The \emph{upper box dimension} of $X$ is  defined by
\[
\ubd X = \inf \{ s : \text{for all $\eps>0$,  } \, \mathcal{L}^d (X_\eps) \lesssim \eps^{d-s}\}
\]
where $E_\eps$ denotes the open $\eps$-neighbourhood of a set $E$.  The \emph{lower box dimension} is defined similarly, but   the requirement that the measure bound holds for all $\eps>0$ is replaced by a sequence of $\eps>0$ tending to zero, that is, 
\[
\lbd X = \inf \{ s : \text{for a sequence of $\eps\searrow 0$,  } \, \mathcal{L}^d (X_\eps) \lesssim \eps^{d-s}\}.
\] 
See \cite{falconer} for more background on the box dimensions, their basic properties, and  alternative formulations.  Here and throughout, $A \lesssim B$ is used to mean that $A \leq cB$ for some constant $c \geq 1$ depending only on fixed parameters.  In particular, $c$ may depend on $X$, the Brascamp--Lieb datum, the ambient spatial dimension $d$, but cannot depend on variable quantities such as $\eps$.

The following is a simple consequence of the Brascamp--Lieb inequality.  It appears to be a fundamental inequality for upper box dimension, which perhaps surprisingly  does not appear in the literature to the best of our knowledge.  That said, one can find antecedents to this inequality in work of J\"{a}rvanp\"{a}\"{a} \cite{jar} (see Corollary \ref{exceptional} later and the discussion beforehand) and even in the original formulation of the Loomis--Whitney inequality  \cite{lw}, which involved volumes and box counts.  Moreover, the well-known product formulae for box dimension are a special case and have been known (and of continued relevance and utility) for many years, see \cite{falconer}.

\begin{thm} \label{main}
Suppose  $P_i, c_i$  ($i = 1, \dots, m$) satisfy \eqref{scaling} and \eqref{dimension} and let $X \subseteq \rd$ be a non-empty bounded set. Then
\[
\ubd X \leq \sum_{i=1}^m c_i\, \ubd P_i(X).
\]
\end{thm}

\begin{proof}
For each $i = 1, \dots, m$, let $s_i> \ubd P_i(X)$. Let $\eps>0$ and  $f_i$ be the indicator function on $P_i(X)_\eps$ for $i = 1, \dots, m$.  Observe that if $ x \in X_\eps$, then $P_i(x) \in P_i(X)_\eps$ for all   projections $P_i$.  Therefore,    
\begin{align*}
\mathcal{L}^d (X_\eps) &\leq \mathcal{L}^d\left(\left\{ x\  : \ \forall i = 1, \dots, m, \  P_i(x) \in P_i(X)_\eps \right\} \right)\\
&  =  \int_{\rd} \prod_{i=1}^m f_i\left(P_i(x)\right)^{c_i} \, dx  \qquad \text{(noting that $f_i=f_i^{c_i}$ since $f_i(y) \in \{0,1\}$)}\\
&\leq C\  \prod_{i=1}^m\left(\int_{\mathbb{R}^{d_i}}  f_i(y) \, dy\right)^{c_i} \qquad \text{(by   the Brascamp--Lieb inequality \eqref{bli})}\\
&\lesssim  \prod_{i=1}^m\mathcal{L}^{d_i} \left( P_i(X)_\eps  \right)^{c_i}\\
&\lesssim  \prod_{i=1}^m \eps^{c_i(d_i-s_i)} \qquad \text{(since   $s_i> \ubd P_i(X)$ for all $i$)}\\
&=  \eps^{\sum_{i=1}^m c_i(d_i-s_i)}\\
&=  \eps^{d - \sum_{i=1}^m c_is_i}  \qquad \text{(by the scaling property \eqref{scaling})}.
\end{align*}
This proves that 
\[
\ubd X \leq \sum_{i=1}^m c_is_i
\]
and letting $s_i \searrow \ubd P_i(X)$ completes the proof.
\end{proof}

The only assumptions on the projections $P_i$ in Theorem \ref{main} are the scaling and dimension assumptions \eqref{scaling} and \eqref{dimension}.  If the conclusion of Theorem \ref{main} is to hold for all bounded sets $X$, then  these assumptions  are clearly necessary and it is perhaps remarkable that they are also sufficient. Indeed, if \eqref{dimension} fails for some $V$,  then a ball in $V$ is a counter-example.  On the other hand, if \eqref{scaling} fails in the sense that
\[
\sum_{i=1}^m c_i d_i <d
\]
then $[0,1]^d$ is a counter-example. 

The special case when $\sum_{i=1}^m d_i = d$ (in which case we are forced to choose $c_i=1$ for all $i$) follows from  the well-known product formula 
\[
\ubd (E_1 \times \cdots  \times E_m) \leq  \ubd E_1 +  \cdots + \ubd E_m,
\]
possibly following a change of basis, see \cite[Chapter 7]{falconer}.  The dimension theory of products is fundamental in fractal geometry and has numerous uses.  As such we expect the more general Theorem \ref{main} to also prove useful.  We give several initial  applications in Section \ref{applications}. 

The following is an immediate corollary, which we believe is new even in this special case.

\begin{cor} \label{simplecor}
Let $X \subseteq \rd$ be a non-empty bounded set. Suppose $P_i$ are the set of orthogonal projections  onto the principal  coordinate  $k$-planes.  Then
\[
\ubd X \leq    \frac{d}{  \binom{d}{k} k } \sum_{i=1}^{\binom{d}{k}}  \ubd P_i(X).
\]
Two further special cases of interest include:
\begin{enumerate}
\item If   $\ubd P_i(X) \leq \alpha$ for all $i$, then
\[
\ubd X \leq  \alpha d/k .
\]
\item If $k=d-1$, then 
\[
\ubd X \leq  \frac{1}{d-1}    \sum_{i=1}^{d}  \ubd P_i(X).
\]
\end{enumerate}
\end{cor}

\begin{proof}
 In this special case $m = \binom{d}{k}$, $d_i = k$ for all $i$,  and we may choose
\[
c_i =\frac{d}{ \binom{d}{k} k} = \frac{(k-1)!(d-k)!}{(d-1)!}.
\] 
The result then follows from Theorem \ref{main}.
\end{proof}

For a given collection of  linear surjections $P_i: \rd \to \mathbb{R}^{d_i}$, the associated \emph{Brascamp--Lieb polytope} is the polytope of weights for which the Brascamp--Lieb constant is finite, see \cite{val}, that is, the set of weights for which $(\{P_i\},\{c_i\})$ is a permissible Brascamp--Lieb datum satisfying \eqref{scaling} and \eqref{dimension}.  Apart from in rather restricted cases including the setting of the Loomis--Whitney inequality where one considers the projections onto the coordinate hyperplanes, the Brascamp--Lieb polytope  is a non-trivial polytope when it is non-empty.  One may leverage the complexity of the Brascamp--Lieb polytope to obtain optimal dimension estimates in Theorem \ref{main}.  One can see that the optimal choice of weights will depend heavily on the relative values of the dimensions $\ubd P_i(X)$.

\begin{cor} \label{maininf}
Let $P_i: \rd \to \mathbb{R}^{d_i}$ be a collection of  linear surjections and let $X \subseteq \rd$ be a non-empty bounded set. Then
\[
\ubd X \leq \inf_{\{c_i\}_i} \sum_{i}  c_i\, \ubd P_i(X)
\]
where the infimum is taken over the Brascamp--Lieb polytope associated with $\{P_i\}_i$.  By convention we interpret $\inf_\emptyset = \infty$.  
\end{cor}

A natural question is whether there is a simple direct and `purely fractal' proof of   Theorem \ref{main}, without going via Brascamp--Lieb.  We conclude this section by pointing out a potential difficulty in a naive (but natural) attempt at proving    Corollary \ref{simplecor} in the case where both (1) and (2) hold (this is the simplest case of Theorem \ref{main} apart from the product formula case---which does admit a simple `fractal' proof).  

 Consider the $(d-1)$-fold self-product   $X^{d-1} \subseteq \mathbb{R}^{d(d-1)}$ and the product $P_1(X) \times  \cdots \times P_d(X) \subseteq \mathbb{R}^{d(d-1)}$.  Then 
\[
\ubd X^{d-1} = (d-1) \ubd X
\]
and 
\[
\ubd  \left( P_1(X) \times  \cdots \times P_d(X) \right)  \leq d \alpha.
\]
Since  $X^{d-1}$ and   $P_1(X) \times  \cdots \times P_d(X)$ live in the same ambient space, that is $\mathbb{R}^{d(d-1)}$,  it is tempting to believe  (after possibly relabelling the coordinates) that 
\[
X^{d-1} \subseteq P_1(X) \times  \cdots \times P_d(X)
\]
in which case the result would follow upon dividing through by $(d-1)$. But this is not true.  For example, consider $X=\{(0,0,0), (1,1,1) \} \subseteq \mathbb{R}^3$.  Then $(0,0,0,1,1,1) \in X^2$ but    $(0,0,0,1,1,1) \notin P_1(X) \times P_2(X) \times P_3(X)$ because each $P_i(X) = \{(0,0),(1,1)\}$ and so all points in $P_1(X) \times P_2(X) \times P_3(X)$ have a different number of zeros  and  ones.

\subsection{Packing  dimension}

The \emph{packing dimension} of a non-empty set $X \subseteq \rd$ is defined by
\[
\pd X = \inf\{ s : X \subseteq \cup_{j \in \mathbb{N}} X_j \text{ with } \ubd X_j \leq s \text { for all $j$}\}.
\]
As such, the packing dimension may be thought of as a countably stable version of the upper box dimension.  It acts as a natural dual to the Hausdorff dimension and is a central concept in fractal geometry. See \cite[Chapter 3]{falconer} for more background on the packing dimension.  Next we show how to establish the Brascamp--Lieb inequality for packing dimension directly from the upper box dimension inequality. As such we also have the analogues of Corollaries \ref{simplecor} and \ref{maininf} but we omit the statements, noting only that the packing dimension results do not require $X$ to be bounded.

\begin{cor} \label{packingcor}
Suppose  $P_i, c_i$  ($i = 1, \dots, m$) satisfy \eqref{scaling} and \eqref{dimension} and let $X \subseteq \rd$ be a non-empty  set. Then
\[
\pd X \leq \sum_{i=1}^m c_i\, \pd P_i(X).
\]
\end{cor}

\begin{proof}
Let $\eps>0$.  For each $i=1, \dots, m$, let
\[
P_i(X) = \bigcup_{j\in \mathbb{N}} Y_{i,j}
\]
where for all $j \in \mathbb{N}$, 
\begin{equation} \label{dinc}
\ubd Y_{i,j} \leq \pd P_i(X) + \eps.
\end{equation}
 Such a decomposition exists directly by the definition of the packing dimension of $P_i(X)$. Further, assume that the sets $Y_{i,j}$ are increasing in $j$, that is, for all $i=1, \dots, m$
\[
Y_{i,1} \subseteq Y_{i,2} \subseteq \cdots.
\]
This additional constraint can be satisfied since the upper box dimension is finitely stable and so we can replace each $Y_{i,j}$ with the set
\[
\bigcup_{j' \leq j} Y_{i,j'}
\]
and the dimension bound still holds.  For each  $j \in \mathbb{N}$, let 
\[
X_{j} := \bigcap_{i=1}^m P_i^{-1}(Y_{i,j}) = \{x \in \rd :   \forall i = 1, \dots, m, \  P_i(x) \in Y_{i,j}\}.
\]
Then, by construction,
\begin{equation} \label{yinc}
P_i(X_{j}) \subseteq Y_{i,j}
\end{equation}
for all pairs $(i,j)$.  Further, observe that 
\begin{equation} \label{xinc}
X \subseteq  \bigcup_{j\in \mathbb{N}} X_{j}.
\end{equation}
To justify this claim, let $x \in X$.  Then, for each $i=1, \dots, m$,
\[
P_i(x) \in P_i(X) = \bigcup_{j\in \mathbb{N}}  Y_{i,j}
\]
and so $P_i(x) \in  Y_{i,j}$ for some $j$.  In particular, $x \in P_i^{-1}(Y_{i,j})$ for some $j$ and so
\[
X \subseteq  \bigcap_{i=1}^m \bigcup_{j\in \mathbb{N}} P_i^{-1}(Y_{i,j}) \subseteq \bigcup_{j\in \mathbb{N}} \bigcap_{i=1}^m  P_i^{-1}(Y_{i,j}) =  \bigcup_{j\in \mathbb{N}} X_{j}
\] 
which proves \eqref{xinc}.  In the above we crucially used that the sets $Y_{i,j}$ (and so also $P_i^{-1}(Y_{i,j}) $) are increasing in $j$ which allows us to exchange intersection and union.

Then,  applying  Theorem \ref{main} to $X_{j}$ gives
\begin{align*}
\ubd X_{j} &\leq \sum_{i=1}^m c_i\, \ubd P_i(X_{j}) \qquad \text{(by  Theorem \ref{main})} \\
&\leq \sum_{i=1}^m c_i\, \ubd Y_{i,j}  \qquad \text{(by \eqref{yinc})} \\
&\leq   \sum_{i=1}^m c_i\, \pd P_i(X) +   \eps d \qquad \text{(by \eqref{dinc} and the scaling property \eqref{scaling}).} \\
\end{align*}
 Therefore,  letting $\eps \to 0$ and recalling the decomposition   \eqref{xinc},
\[
\pd X \leq \sum_{i=1}^m c_i\, \pd P_i(X),
\]
as required.
\end{proof}

\subsection{Assouad  dimension}

Another important notion of fractal dimension is the Assouad dimension.  It is sensitive to local information and often displays counter-intuitive properties.  We briefly recall the definition and refer the reader to \cite{fraser} for more details.  Given a non-empty set $X \subseteq \rd$, the \emph{Assouad dimension} of $X$ is defined by
\[
\ad X = \inf\left\{ s : \exists c>0, \, \forall x \in X, \, \forall R>r>0,  \ N_r\left(B(x,R) \cap X\right)  \leq c  \left(   \frac{R}{r}   \right)^{s}  \right\}
\]
where, for a bounded set $F \subseteq \mathbb{R}^d$, $N_r(F)$ denotes the smallest number of sets of diameter $r$ required to cover $F$. In particular, the constant $c$  in the definition of Assouad dimension may depend on $d$, $s$,  and $X$ but not on $x\in X$ or on the scales $R>r>0$. 

 The Assouad dimension also turns out to satisfy the Brascamp--Lieb inequality. In fact, the same proof also shows that the \emph{Assouad spectrum} satisfies the Brascamp--Lieb inequality, but we leave the details to the reader.  See \cite{fraser} for the definition of the Assouad spectrum.

\begin{thm} \label{assouad}
Suppose  $P_i, c_i$  ($i = 1, \dots, m$) satisfy \eqref{scaling} and \eqref{dimension} and let $X \subseteq \rd$ be a non-empty  set. Then
\[
\ad X \leq \sum_{i=1}^m c_i\, \ad P_i(X).
\]
\end{thm}

\begin{proof}
For each $i = 1, \dots, m$, let $s_i> \ad P_i(X)$. Let $R>r>0$, $x \in X$  and  $f_i$ be the indicator function on $\big(P_i(B(x,R) \cap X)\big)_r$ for $i = 1, \dots, m$.  Observe that if $ z \in \big(B(x,R) \cap X\big)_r$, then $P_i(z) \in \big(P_i(B(x,R) \cap X)\big)_r$ for all   projections $P_i$.  Moreover, observe that for all non-empty bounded sets $F$ in $\mathbb{R}^k$
\begin{equation} \label{convert}
N_r(F) \approx_k r^{-k} \mathcal{L}^{k}(F_r)
\end{equation}
with the implicit constants depending on $k$ but \emph{not} on $F$. We apply this below with $F$ depending on $R$ and   it is essential  that the implicit constants do not depend on $R$.  To justify \eqref{convert} first note that an $r$-cover of $F$ can be upgraded to a cover of $F_r$ by replacing each covering set by an appropriate  ball of  diameter $2r$ and so 
\[
\mathcal{L}^{k}(F_r) \lesssim_k r^{k} N_r(F).
\]
On the other hand, suppose there is a maximal collection of disjoint closed balls of diameter  $r$ centred in $F$ of cardinality  $M$.  Replacing each of these balls by a ball with the same centre but twice the diameter yields a $2r$-cover of $F$.  But then
\[
\mathcal{L}^{k}(F_r) \gtrsim_k r^{k} M \geq r^k  N_{2r}(F) \gtrsim_k r^k N_r(F),
\] 
as required. Therefore,    
\begin{align*}
N_r\left(B(x,R) \cap X\right) &\approx r^{-d} \, \mathcal{L}^d \left( \big(B(x,R) \cap X\big)_r \right)   \qquad \text{(by \eqref{convert})} \\
&\leq r^{-d} \,  \mathcal{L}^d\left(\left\{ z\  : \ \forall i = 1, \dots, m, \  P_i(z) \in \big(P_i(B(x,R) \cap X)\big)_r \right\} \right)\\
&  = r^{-d} \,  \int_{\rd} \prod_{i=1}^m f_i\left(P_i(z) \right)^{c_i} \, dz \qquad \text{(noting that $f_i=f_i^{c_i}$ since $f_i(y) \in \{0,1\}$)}\\
&\leq r^{-d} \,  C\  \prod_{i=1}^m\left(\int_{\mathbb{R}^{d_i}}  f_i(y) \, dy\right)^{c_i} \qquad \text{(by   the Brascamp--Lieb inequality \eqref{bli})}\\
&\lesssim r^{-d} \,  \prod_{i=1}^m\mathcal{L}^{d_i} \left( \big(P_i(B(x,R) \cap X)\big)_r  \right)^{c_i}\\
&\approx  r^{-d} \,  \prod_{i=1}^m  \left( r^{d_i} \, N_r\left(P_i(B(x,R) \cap X)\right)  \right)^{c_i}\qquad \text{(by \eqref{convert})} \\
&=    \prod_{i=1}^m   N_r\left(P_i(B(x,R) \cap X)\right)^{c_i}  \qquad \text{(by the scaling property \eqref{scaling})} \\
&\leq    \prod_{i=1}^m     N_r\left( B(P_i(x),R) \cap P_i(X)\right)^{c_i}   \\
&\lesssim    \prod_{i=1}^m  \left(   \frac{R}{r}   \right)^{s_ic_i}  \qquad \text{(since $s_i > \ad P_i(X)$ for all $i$)} \\
&=    \left(   \frac{R}{r}   \right)^{\sum_is_ic_i}  .
\end{align*}
This proves that 
\[
\ad X \leq \sum_{i=1}^m c_is_i
\]
and letting $s_i \searrow \ad P_i(X)$ completes the proof.
\end{proof}

\subsection{Failure for lower box dimension and Hausdorff  dimension}

The Brascamp--Lieb inequality does \emph{not} hold for lower box dimension or Hausdorff dimension.  Indeed, this is immediately seen to be true by the well-known examples satisfying
\[
\hd A \times B > \hd A + \hd B
\]
with the best general upper bound in such situations being
\[
\hd A \times B \leq  \hd A + \pd B,
\]
see \cite[Chapter 7]{falconer}. Analogous examples exist for the lower box dimension. The basic idea behind these well-known examples is that $A$ and $B$ are constructed such that they alternate between looking large on some scales (say, 1-dimensional)  and small on others (say, 0-dimensional)   but that at any given scale at least one of $A$ or $B$ looks large (1-dimensional).  This essentially ensures that $A$ and $B$ are both small, that is, they have Hausdorff or lower box dimension 0 but their product is large, that is, has Hausdorff or lower box dimension 1.  This construction can be generalised to the Brascamp--Lieb setting as follows. We leave the precise details to the reader. We construct compact sets $A_i \subseteq [0,1]$ for $i = 1, \dots, d$ with the property that each set $A_i$ looks 1-dimensional on a sparse sequence of scales tending to zero but otherwise looks 0-dimensional.  Moreover, we interweave the relevant scales such that  at any given scale precisely  one of the $A_i$ looks 1-dimensional.  Let $X=A_1 \times \cdots \times A_d$.  It can be arranged such that 
\[
\hd A_i = 0 \qquad \text{(for all $i=1, \dots, d$),}
\]
\[
\hd X = 1
\] 
and
\[
\hd P(X) = 0
\]
for all principal coordinate projections $P$ onto the $d$ many $(d-1)$-dimensional coordinate  subspaces.  Indeed, for each such projection, one of the sets $A_i$ is suppressed and so the scales at which that particular $A_i$ was the only large factor will witness that the projection is 0-dimensional.  In particular, there can be no Brascamp--Lieb inequality.

By inspecting the proof of Theorem \ref{main} we get the following.
\begin{prop}  
Let $X \subseteq \rd$ be a non-empty bounded set. Then
\[
\lbd X \leq c_j\, \lbd P_j(X) + \sum_{i=1: i \neq j}^m c_i\, \ubd P_i(X)
\]
for all $j =1, \dots, m$.
\end{prop}

\begin{proof}
The proof is the same as the proof of Theorem \ref{main} apart from that we choose a sequence of $\eps \to 0$ which picks out the lower box dimension of $P_j(X)$ and we are forced to make do with the upper box dimension estimates for $i \neq j$.
\end{proof}

One may also recover a Brascamp--Lieb inequality for lower box dimension for `self-products', see Corollary \ref{self-prod}.

\section{Main results:~nonlinear Brascamp--Lieb inequalities for fractal dimensions}

In this section we obtain   nonlinear Brascamp--Lieb inequalities for the upper box, packing and Assouad  dimensions.  The main difficulty of using the nonlinear Brascamp--Lieb inequality \eqref{nbli} in the proof of Theorem \ref{main} is that it is local, that is, it only holds for the neighbourhood $U$ of $x_0$ and we cannot ensure that $X \subseteq U$.  The solution to this is to take finite covers of $X$ by domains $U$ and then assume uniformity of the Brascamp--Lieb datum over $X$.  More general formulations are possible, for example, the uniformity does not need to be over the whole of $X$ but we only need $X$ to be covered by a finite collection of neighbourhoods of points which enjoy such uniformity, but we leave such extensions to the reader.

\begin{thm} \label{mainnl}
Let $m,d \geq 1$ be natural numbers and, for each $i=1, \dots, m$, let $d_i \leq d$ be a natural number and $c_i>0$ be a weight such that
\begin{equation} \label{scaling2}
\sum_{i=1}^mc_i d_i = d.
\end{equation}  
Let $X \subseteq \rd$ be a non-empty compact set. Further, for $i=1, \dots, m$ let $T_i: \rd \to \mathbb{R}^{d_i}$ be $C^2$ submersions in a common neighbourhood of $X$.  For all $x_0 \in X$,  write $J_{x_0}(T_i) = P_i^{x_0}$  where each $P_i^{x_0}: \rd \to \mathbb{R}^{d_i}$ is a  linear surjection.  Assume that, for all $x_0 \in X$ and all subspaces $V$ of $\rd$,
\begin{equation} \label{dimension2}
\dim V \leq \sum_{i=1}^m c_i \dim \left(P_i^{x_0}(V)\right).
\end{equation}  
 Then
\[
\ubd X \leq \sum_{i=1}^m c_i\, \ubd T_i(X),
\]
 \[
\pd X \leq \sum_{i=1}^m c_i\, \pd T_i(X),
\]
and 
\[
\ad X \leq \sum_{i=1}^m c_i\, \ad T_i(X).
\]
\end{thm}

\begin{proof}
We first consider the upper box dimension inequality. For all $x_0 \in X$, let $U(x_0)$ be a neighbourhood of $x_0$ such that 
 \begin{equation} \label{nbli2}
\int_{U(x_0)} \prod_{i=1}^m f_i\left(T_i(x)\right)^{c_i} \, dx \leq 2C_{x_0}\  \prod_{i=1}^m\left(\int_{\mathbb{R}^{d_i}}  f_i(y) \, dy\right)^{c_i},
\end{equation}  
where $C_{x_0}$ is the Brascamp--Lieb constant associated to the Brascamp--Lieb datum $(\{P_i^{x_0}\}_i, \{c_i\}_i)$ which \emph{a priori} depends on $x_0$.  Nevertheless, each $C_{x_0}$ is a finite constant.  The inequality \eqref{nbli2} follows from the nonlinear Brascamp--Lieb inequality  \eqref{nbli} in the case $\eps=1$.  Since $X$ is compact, we may extract a finite subcover of $X$ from the sets $U({x_0})$ and we denote this finite subcover by $\{U(x_0)\}_{x_0 \in X_0}$ where $X_0$ is a finite subset of $X$.  Let
\begin{equation} \label{c0choice}
C_0 = \max_{x_0 \in X_0} C_{x_0}<\infty.
\end{equation}
Let $\eps>0$ be sufficiently small such that 
\begin{equation} \label{decompy2}
X_{\eps} \subseteq \bigcup_{x_0 \in X_0}  U(x_0)
\end{equation} 
and decompose 
\begin{equation*}
X_\eps = \bigcup_{x_0 \in X_0} X_\eps \cap U(x_0)
\end{equation*}
noting that 
\[
\mathcal{L}^d (X_\eps)  \leq \sum_{x_0 \in X_0} \mathcal{L}^d (X_\eps \cap U(x_0)) .
\]
We now follow the proof of the linear Brascamp--Lieb inequality for upper box dimension (that is, Theorem \ref{main})  for each of the sets $X_\eps \cap U(x_0)$ separately. Fix $x_0 \in X_0$ and, for each $i = 1, \dots, m$, let
\[
s_i> \ubd T_i(X)  \geq \ubd T_i(X \cap U(x_0)).
\]
 Let  $f_i$ be the indicator function on $ T_i(X \cap U(x_0))_{\eps'}$ for $\eps' = \|T_i\| \eps$ for $i = 1, \dots, m$. Here
\[
\|T_i\| = \sup_{\substack{x,y \in U:\\ x \neq y}} \frac{|T_i(x)-T_i(y)|}{|x-y|}
\]
is the Lipschitz norm of $T_i$. Note that each of the submersions $T_i$  is $C^2$ on a neighbourhood of $X$ and so we may assume $\| T_i\| <\infty$ by considering a possibly smaller neighbourhood of $X$ if necessary.  Moreover, we only consider finitely many submersions $T_i$ and so we may safely use the estimate $\| T_i\| \approx 1$. 

Observe that if $ x \in X_\eps \cap U(x_0)$, then $T_i(x) \in T_i(X \cap U(x_0))_{\eps'}$  for all   submersions  $T_i$.  Therefore,    
\begin{align*}
\mathcal{L}^d (X_\eps \cap U(x_0)) &\leq \mathcal{L}^d\left(\left\{ x \in U(x_0)\  : \ \forall i = 1, \dots, m, \  T_i(x) \in T_i(X \cap U(x_0))_{\eps'} \right\} \right)\\
&  =  \int_{U(x_0)} \prod_{i=1}^m f_i\left(T_i(x)\right)^{c_i} \, dx \qquad \text{(noting that $f_i=f_i^{c_i}$ since $f_i(y) \in \{0,1\}$)}\\
&\leq 2C_0\  \prod_{i=1}^m\left(\int_{\mathbb{R}^{d_i}}  f_i(y) \, dy\right)^{c_i} \qquad \text{(by   \eqref{nbli2} and \eqref{c0choice})}\\
&\lesssim  \prod_{i=1}^m\mathcal{L}^{d_i} \left( T_i(X \cap U(x_0))_{\eps'}  \right)^{c_i}\\
&\lesssim  \prod_{i=1}^m (\|T_i\| \eps)^{c_i(d_i-s_i)} \qquad \text{(since   $s_i> \ubd T_i(X \cap U(x_0))$ for all $i$)}\\
&\approx  \eps^{\sum_{i=1}^m c_i(d_i-s_i)} \\
&=  \eps^{d - \sum_{i=1}^m c_is_i}  \qquad \text{(by  \eqref{scaling2})}.
\end{align*}
 This proves that 
 \[
 \mathcal{L}^d (X_\eps)  \leq \sum_{x_0 \in X_0} \mathcal{L}^d (X_\eps \cap U(x_0)) \lesssim \eps^{d - \sum_{i=1}^m c_is_i}
 \]
 and therefore
\[
\ubd X \leq \sum_{i=1}^m c_is_i.
\]
Letting $s_i \searrow \ubd T_i(X) $   completes the proof of the upper box dimension estimate.

The packing dimension inequality follows from the upper box dimension inequality in the same way as Corollary \ref{packingcor} follows from Theorem \ref{main}. Indeed, we did not use linearity in the proof of Corollary \ref{packingcor}.   Finally, the Assouad dimension inequality follows via a similar strategy to the upper box dimension proof above but adopting the localised approach of Theorem \ref{assouad}.   More specifically, consider the same finite cover of $X$ as above and the same $\eps>0$, that is, we begin with \eqref{decompy2}.  Choose $0<r<R<\eps$ and $x \in X$ and consider the sets 
\[
(B(x,R) \cap X)_r \cap U(x_0)
\]
separately and afterwards appeal to
\[
 N_r(B(x,R) \cap X) \approx r^{-d} \mathcal{L}^d((B(x,R) \cap X)_r) \lesssim \sum_{x_0 \in X_0} r^{-d} \mathcal{L}^d((B(x,R) \cap X)_r \cap U(x_0)).
\]
Then,  for each $x_0 \in X_0$, following the proof of Theorem \ref{assouad} and the argument above, we get
\[
r^{-d}\mathcal{L}^d((B(x,R) \cap X)_r \cap U(x_0)) \lesssim \left(\frac{R}{r} \right)^{\sum_i s_i}
\] 
for $s_i>\ad T_i(X)$ chosen arbitrarily. We leave the precise details to the reader. 
\end{proof}



\section{Applications} \label{applications}

We now  give some simple applications of the Brascamp--Lieb inequalities for fractal dimensions to help motivate this line of enquiry. The results in this section also hold (for Borel but not necessarily bounded sets) with upper box  dimension (and upper box  dimension profiles)   replaced with packing dimension (and packing dimension profiles).  Some of the results also apply to Assouad dimension but one has to be slightly careful because Assouad dimension has some different properties, including not being  Lipschitz stable in general.  We only state the results for upper box  dimension to simplify the exposition. 

\subsection{Exceptional sets and dimension profiles}

Here and throughout $G(d,k)$ denotes the Grassmannian manifold consisting of the set of $k$-dimensional subspaces of $\rd$ and $P_V$ denotes orthogonal projection onto $V$. 

One way of interpreting the Brascamp--Lieb formula for fractal dimensions is that there cannot be too many projections with very small dimension. Indeed, in the setting of Cartesian products, it is a well-known folklore result that, given a  compact set $X \subseteq \mathbb{R}^2$, there is at most one $V \in G(2,1)$ such that
\[
\ubd P_{V}(X) < \frac{\ubd X}{2}.
\]
Indeed, if there were two such spaces $V, V'$, then (after a possible change of basis) $X \subseteq P_V(X) \times P_{V'}(X)$ and 
\[
\ubd X \leq \ubd P_V(X) + \ubd P_{V'}(X) < \ubd X,
\]
a contradiction.   The following result generalises this idea and also generalises \cite[Theorem 3.4]{jar} which asserts the existence of a coordinate subspace $V \in G(d,k)$ such that $\ubd P_{V}(X) \geq \frac{k}{d} \, \ubd X$.  We say a collection of linear surjections $P_i: \rd \to \mathbb{R}^{d_i}$  is \emph{BL-feasible} if there exist  associated weights $\{c_i\}$ for which $(\{P_{i}\}, \{c_i\})$ is an admissible Brascamp--Lieb datum satisfying \eqref{scaling} and \eqref{dimension}, that is, if the associated Brascamp--Lieb polytope is non-empty. For example, the collection of principal coordinate projections onto $k$-planes is BL-feasible but BL-feasible collections can be much smaller in general.  The problem of determining when a given collection is BL-feasible is quite complicated in general, see, for example, \cite{val}.
\begin{cor} \label{exceptional}
Let $X \subseteq \rd$ be a non-empty bounded set and let $\{P_i\}$  be a BL-feasible collection  of linear surjections with $P_i: \rd \to \mathbb{R}^{d_i}$.  Then
\[
\ubd P_{i}(X) \geq \frac{d_i}{d} \, \ubd X
\]
for some $i$.
\end{cor}
\begin{proof}
Suppose $\{P_i\}$ is   a BL-feasible collection and choose permissible associated weights $\{c_i\}$ arbitrarily. Suppose $\ubd P_{i}(X) < \frac{d_i}{d} \, \ubd X$ for all $i$.  But then Theorem \ref{main} implies
\[
\ubd X \leq   \sum_{i} c_i  \,  \ubd P_{i}(X) <  \frac{\ubd X}{d} \,  \sum_{i} c_i d_i    = \ubd X,
\]
by \eqref{scaling}, and this is a contradiction. 
\end{proof}

Understanding the `typical behaviour' of dimension under orthogonal projection is an important and active area in fractal geometry, see \cite{falconer, fraser, mattila}.  The seminal result is the Marstrand--Mattila projection theorem which says that for a Borel set $X \subseteq \rd$, 
\[
\hd P_V(X) = \min\{\hd X , k\}
\]
 for `almost all' $V \in G(d,k)$ with respect to the Grassmannian measure.  The box and packing dimensions also satisfy a `Marstrand theorem' in the sense that $\ubd P_V(X)$ assumes a constant value for almost all $V \in G(d,k)$, but this constant value is more subtle than in the case of the Hausdorff dimension.  In fact it is given by the $k$-dimensional upper box (or packing or lower box) dimension profile of $X$, denoted by  $\ubd^k X$ and introduced in \cite{FH2}.  There is interest in the exceptional set, that is,
\[
E_{d,k}(s):=\{ V \in G(d,k): \ubd P_V(X) < s\}.
\]
In particular, for $s\leq\ubd^k X$, $E_{d,k}(s)$ is a null set with respect to the Grassmannian measure, but can still be large in the sense of Hausdorff dimension.  The Brascamp--Lieb inequalities can be used to provide new structural information concerning certain exceptional sets.  
\begin{cor}
Let $1 \leq k < d$ be integers and   $X \subseteq \rd$ be a non-empty bounded set.  The exceptional set 
\begin{equation} \label{eset}
E_{d,k}\left(\frac{k}{d} \, \ubd X\right)
\end{equation}
 does not contain a subset for which the associated orthogonal projections are BL-feasible. In particular, the exceptional set \eqref{eset} is not dense in $G(d,k)$.
\end{cor}

\begin{proof}
If there was such a subset of \eqref{eset}, then this would contradict Corollary \ref{exceptional}.  Further, for a fixed collection of weights $\{c_i\}$, the  set of surjections $\{P_i\}$ such that the associated Brascamp--Lieb constant is finite is open \cite{cts}.  Therefore, if the exceptional set \eqref{eset} was dense in $G(d,k)$, then we could find a BL-feasible subset by approximating a known finite BL-feasible collection, such as the set of $k$-dimensional coordinate projections.  
\end{proof}

Since full measure sets are dense, the previous corollary gives a simple bound for the dimension profiles.    This bound is not new; it was observed in \cite{jar} and improved to the best possible in \cite{FH}, see below.

\begin{cor}
Let $X \subseteq \rd$ be a non-empty bounded set and $1 \leq k < d$. Then
\[
\ubd^k X \geq \frac{k}{d} \, \ubd X.
\]
\end{cor}

In general, better bounds for $\ubd^k X $ are known.  In fact
\[
\frac{\ubd X}{1+(1/k-1/d)\ubd X} \leq \ubd^k X  \leq \min\{\ubd X, k\}
\]
and these inequalities are the best possible in general, see \cite{FH,FH2}. It is not surprising that the approach above does not yield sharp bounds for the dimension profile because we are proving something rather stronger.  Rather than giving almost sure lower bounds for the dimension, we are giving bounds which hold outside of a much smaller and more structurally constrained set of exceptions. We believe rather stronger conclusions about the exceptional sets are possible along these lines, but would require careful analysis of the structure  of BL-feasible sets, which in general is tricky.  Thanks to a simple description from \cite[Theorem 1.3]{val} in the rank $d-1$ case we get the following.  

\begin{cor}
Let $d \geq 3$ be an integer and   $X \subseteq \rd$ be a non-empty bounded set.  The exceptional set 
\begin{equation} \label{eset2}
E_{d,d-1}\left(\frac{d-1}{d} \, \ubd X\right)
\end{equation}
is necessarily contained in a co-dimension 1 submanifold of $G(d,d-1)$ and, in particular, has Hausdorff dimension at most $d-2 = \dim G(d,d-1)-1$. 
\end{cor} 

\begin{proof}
It follows from \cite[Theorem 1.3]{val} that any collection of $V$ for which $\rd$ decomposes as a direct sum of a subset of the kernels $\textup{ker}(P_V)$ is BL-feasible. Any collection   of $V$ from $G(d,d-1)$ which is not contained in a co-dimension 1 submanifold therefore necessarily contains a BL-feasible subset and therefore if the exceptional set \eqref{eset2} is not contained in a co-dimension 1 submanifold, then we contradict Corollary \ref{exceptional}.
\end{proof}

\subsection{Building up, not down}

So far we have been adopting the viewpoint that we are given a set $X \subseteq \rd$ and want to make deductions about the dimensions of $X$ based on its projections.  Analogous to the dimension theory of Cartesian products, we may view the problem in reverse.  That is, we are given a collection of sets $X_i \subseteq \mathbb{R}^{d_i}$ and build a set $X \subseteq \rd$ from this collection.  More precisely, given $m,d \geq 1$ and  a collection of linear surjections $P_i: \mathbb{R}^d \to \mathbb{R}^{d_i}$ ($i = 1, \dots, m$) and a collection of sets $X_i \subseteq \mathbb{R}^{d_i}$ define
\[
\bigotimes_{i=1}^m X_i : = \{ x \in \mathbb{R}^d : \forall i=1, \dots, m, \  P_i(x) \in X_i\} = \bigcap_{i=1}^m P_i^{-1}(X_i).
\]
That is, $\bigotimes_{i=1}^m X_i $ is the largest set by inclusion with prescribed images $P_i(X)\subseteq X_i$.  Note that it may be empty if the sets $X_i$ are `incompatible'.  Moreover, if $\sum d_i = d$ and $\cap_{i=1}^m \textup{ker}(P_i) = \{0\}$, then
\[
\bigotimes_{i=1}^m X_i = X_1 \times \cdots \times X_m
\]
is the usual Cartesian product. The following result is immediate from Theorem \ref{main}.

\begin{cor} \label{constrainedproduct}
Let $\{P_i\}$ and  $\bigotimes_{i=1}^m X_i $ be defined as above.  Then, for any choice of weights $\{c_i\}$ such that $P_i, c_i$ satisfy \eqref{scaling} and \eqref{dimension},
\[
\ubd \bigotimes_{i=1}^m X_i  \leq \sum_{i=1}^m c_i \, \ubd X_i.
\]
\end{cor}

Since we are now building up, not down, if we let the `pieces' $X_i$ all be the same set, then the `self-product' will admit a Brascamp--Lieb inequality for lower box dimension.  There is no such result for Hausdorff dimension, as it is known to fail even in the Cartesian product case. 

\begin{cor} \label{self-prod}
Let $X \subseteq \mathbb{R}^l$ for some $l <d$ and, for $i=1, \dots, m$, let $P_i:\rd \to  \mathbb{R}^l$ be a family of linear surjections.  Then, for any choice of weights $\{c_i\}$ such that $P_i, c_i$ satisfy \eqref{scaling} and \eqref{dimension},
\[
\lbd \bigotimes_{i=1}^m X  \leq \sum_{i=1}^m c_i \, \lbd X .
\]
\end{cor}

\begin{proof}
Inspecting the proof of Theorem \ref{main} we are able to obtain this improved estimate by bounding the volume of the $\eps$-neighbourhood of $\bigotimes_{i=1}^m X$ for a sequence of $\eps \searrow 0$ since the $m$ different estimates we need are all referring to the same set at the same scale.  In particular, we choose the sequence of $\eps$ to align with where $X$ is small, thus picking up the lower box dimension.
\end{proof}

\subsection{Constrained sumsets}

Given $E \subseteq [0,1]$, the sumset 
\[
E+E=\{x+y: x,y \in E\}
\]
is an important object in additive combinatorics.  The sumset $E+E$ is a Lipschitz image of the product set $E \times E$ and therefore it is straightforward from the product formula that
\[
\ubd  (E+E) \leq \ubd  (E \times E) \leq 2 \ubd  E.
\]
Here we consider a constrained sumset problem. Let $X \subseteq [0,1]^2$ and define
\[
\Delta(X) = \{x+y+z : (x,y), (y,z), (x, z) \in X\} \subseteq \mathbb{R}
\]
and
\[
\tilde X = \{(x,y,z) : (x,y), (y,z), (x, z) \in X\} \subseteq \mathbb{R}^3.
\]
Then $\Delta(X)$ is a Lipschitz image of $\tilde X$ and so
\[
\ubd  \Delta(X) \leq \ubd  \tilde X \leq \frac{3 \ubd  X}{2}
\]
by Theorem \ref{main}.  More generally, we have the following result.  Further variants are also possible, but we leave these to the reader.

\begin{cor} \label{sumset}
Let $d >k \geq 1$ be integers and $X \subseteq \mathbb{R}^k$ and define
\[
\Delta(X) = \left\{ \sum_{i=1}^d z_i : P(z_1, \dots, z_d) \in X \text{ for all $k$-dimensional coordinate projections $P$}\right\}.
\]
Then
\[
\ubd \Delta(X) \leq \frac{d}{k} \, \ubd X.
\]
In particular, if $\ubd X < k/d$ then $\Delta(X)$ has zero Lebesgue measure. 
\end{cor}

\begin{proof}
The constrained sumset $\Delta(X)$ is a Lipschitz image of the set
\[
\bigotimes_{i=1}^{\binom{d}{k}} X = \left\{ z \in \rd :  \forall i=1, \dots, \binom{d}{k}, \  P_i(z) \in X\right\}
\] 
under the map $(z_1, \dots, z_d ) \mapsto z_1 + \cdots + z_d$ where the $P_i$ are the $\binom{d}{k}$ many principal  $k$-dimensional coordinate projections. Then, applying Corollary \ref{constrainedproduct},
\[
\ubd \Delta(X) \leq \ubd \bigotimes_{i=1}^{\binom{d}{k}} X  \leq \frac{d}{k} \,  \ubd X,
\]
as required.
\end{proof}

The estimate from Corollary \ref{sumset} is sharp.  For example, let $X = E \times \cdots \times E$ be the $k$-fold Cartesian product of a set $E \subseteq [0,1]$ with upper box  and Fourier dimension (denote by $\fd$) equal to $s \in [0,1/d]$.  Such sets $E$   can be constructed, for example, by taking the image of a self-similar set in $[0,1]$ of (upper box and Hausdorff) dimension $s/2$ under $1$-dimensional Brownian motion $B: [0,\infty) \to \mathbb{R}$. The image is almost surely a Salem set of (Fourier) dimension $s$ by \cite[Theorem 1, Chapter 17]{kahane} and, moreover, the upper box dimension of the image cannot exceed $s$ since $1$-dimensional Brownian motion  is almost surely $\alpha$-H\"older for all $0<\alpha<1/2$. Then
\[
\ubd X  = ks
\]
and it is straightforward to see that the product structure of $X$ ensures that
\[
\Delta(X) = E+ \cdots + E
\]
is the $d$-fold sumset of $E$.  Then
\[
\ubd \Delta(X) \geq \hd (E+ \cdots + E) \geq \fd (E+ \cdots + E)  \geq d \fd E =  d s
\]
and 
Corollary \ref{sumset} gives
\[
\ubd \Delta(X) \leq \frac{d}{k} \, \ubd X  = ds.
\]
In the above we used that the Fourier dimension is super-additive with respect to taking sumsets.  This is an easy consequence of the convolution formula for the Fourier transform, see \cite{mattila}.  In particular, see the proof of  \cite[Proposition 3.14]{mattila}.

\subsection{Nonlinear examples}

Here we collect some simple consequences of the nonlinear Brascamp--Lieb inequality for upper box dimension, Theorem \ref{mainnl}.  Rather than try to make very general statements we focus on simple examples.
\\ \\
\emph{Example 1:~dimension bounds from partial directional information}.  Let $X \subseteq (0,\infty)^3$ be compact and define three nonlinear $C^2$ submersions   $T_i:(0,\infty)^3 \mapsto \mathbb{R}^2$ by
\[
T_1(z) = |z|(z_2,z_3), \qquad T_2(z) = |z|(z_1,z_3), \qquad T_3(z) = |z|(z_1,z_2).
\]
By direct calculation of the Jacobians we find that these submersions are all of rank 2 at all points in their domain.  The maps $T_i$ record the  modulus and also some partial directional information about the input $z$. It follows from Theorem \ref{mainnl} that
\[
\ubd X \leq \frac{1}{2} \sum_{i=1}^3 \ubd T_i(X).
\]
\emph{Example 2:~nonlinear constrained sumsets}.  Let $X \subseteq (0,\infty)^2$ be compact and define
\[
\Delta(X) = \{x+y+z : (xy,xz), (xy,yz), (xz, yz) \in X\} \subseteq \mathbb{R}
\]
and
\[
\tilde X = \{(x,y,z) : (xy,xz), (xy,yz), (xz, yz) \in X\} \subseteq \mathbb{R}^3.
\]
Further, define three nonlinear $C^2$ submersions  $T_i: (0,\infty)^3 \mapsto \mathbb{R}^2$ by
\[
T_1(x,y,z) = (xy,xz), \qquad T_2(x,y,z) = (xy, yz), \qquad T_3(x,y,z) = (xz, yz).
\]
Again, these submersions are easily seen to be rank 2 everywhere in their domain. Then each $T_i$ maps $\tilde X$ into $X$.  It follows from Theorem \ref{mainnl} that
\[
\ubd \tilde X \leq \frac{1}{2} \sum_{i=1}^3 \ubd T_i( \tilde X) \leq  \frac{3}{2} \  \ubd X.
\]
But  $\Delta(X)$ is a Lipschitz image of $\tilde X$ and so
\[
\ubd  \Delta(X) \leq \ubd  \tilde X \leq \frac{3}{2} \ \ubd X.
\]
 \emph{Example 3:~there cannot be many  small radial projections}.   Given $z \in \rd$, the \emph{radial projection} associated to the `pin' $z$ is the map $R_z : \rd\setminus\{z\} \mapsto S^{d-1}$ defined by
 \[
R_z(x) = \frac{x-z}{|x-z|}.
 \]
 There is interest in understanding how dimension behaves under radial projections, for example, relating $\ubd R_z(X)$ to $\ubd X$ for generic pins $z$.   Radial projections are nonlinear and we can use our nonlinear Brascamp--Lieb inequality to show that in generic situations there cannot be too many small radial projections.

 \begin{cor} \label{exceptionalradial}
Let $X \subseteq \rd$ be a non-empty compact set and let $Z \subseteq \rd$ be a collection of $d$ many points in general position and assume that $X$ does not intersect any affine hyperplanes spanned by proper subsets of $Z$.  Then 
\[
\ubd X \leq \frac{1}{d-1} \sum_{z \in Z} \ubd R_z(X)
\]
and, in particular, there must exist $z \in Z$ such that 
\[
\ubd R_{z}(X) \geq \frac{d-1}{d} \, \ubd X.
\]
\end{cor}
\begin{proof}
First consider a partition of $X$ into finitely many pieces $X_j$ each with very small diameter, say $|X_j| \leq 1/10$.  It suffices to prove the claim for each piece $X_j$ and appeal to finite stability of upper box dimension.  For each $z \in Z$, consider $R_z(X_j) \subseteq S^{d-1}$.  Since $|R_z(X_j)| \leq 1/10$, there is a smooth global chart defined in a neighbourhood of $R_z(X_j)$ and therefore we can find   a $C^2$ diffeomorphism  $T_z$ defined on a neighbourhood of $R_z(X_j)$ which maps into $\mathbb{R}^{d-1}$ and the map $T_z \circ R_z$ is now a  $C^2$ submersion of rank $(d-1)$  defined in a neighbourhood of $X_j$.  Looking to apply Theorem \ref{mainnl}, we may choose $c_i = \frac{1}{d-1}$ for all $i = 1, \dots, |Z| = d$.  In particular,   \eqref{scaling2} is satisfied.  Moreover,  the fact that the points in $Z$ are in general position and that $X$ does not intersect hyperplanes spanned by proper subsets of $Z$ ensures    \eqref{dimension2} is  satisfied for all $x_0 \in X$.  Then Theorem \ref{mainnl} implies that 
\[
\ubd X_j \leq \frac{1}{d-1} \sum_{z \in Z} \ubd (T_z \circ R_z)(X_j ) = \frac{1}{d-1} \sum_{z \in Z} \ubd R_z(X_j )
\]
and the first claim follows. Finally, suppose that for all $z \in Z$  
\[
\ubd R_{z}(X) < \frac{d-1}{d} \, \ubd X.
\]
But then
\[
\ubd X \leq \frac{1}{d-1} \sum_{z \in Z} \ubd R_z(X) < \ubd X , 
\]
a contradiction. 
\end{proof}

The assumptions in Corollary \ref{exceptionalradial} cannot be removed in general. For example, consider a collection $Z \subseteq \rd $ of $d$ many pins in general position and let $X$ be a unit line segment which is collinear with (but not containing) two of the points from $Z$.  Then, for these two pins,
\[
\ubd R_z(X) = 0
\]
and for all the others 
\[
\ubd R_z(X) = 1.
\]
But then the upper bound from Corollary  \ref{exceptionalradial} would give 
\[
\frac{d-2}{d-1}<1 = \ubd X 
\]
and so the conclusion fails. 

\section*{Acknowledgements}

I thank Jonathan Bennett for his interesting talk about the Brascamp--Lieb inequality during the 2024 Madison Lectures in Harmonic Analysis. His talk, and the discussions thereafter on the Madison Promenade,  inspired this work. I am also grateful to Kenneth Falconer for pointing out the reference \cite{jar} and to Jonathan Bennett for more recent discussions.

\end{document}